\documentclass[11pt, a4paper]{amsart}

\usepackage{amsmath,amssymb,amsthm}
\usepackage{hyperref}
\usepackage[dvipsnames,usenames]{color}
\usepackage[latin1]{inputenc}
\usepackage{graphicx}
\usepackage{psfrag}
\usepackage{color}

\newtheorem{guia}{}

\newtheorem{teorema}[guia]{Theorem}

\newtheorem{lema}[guia]{Lemma}

\newcommand{\al}{\alpha}\newcommand{\be}{\beta}

\newcommand{\De}{\Delta}

\newcommand{\la}{\lambda}

\newcommand{\Om}{\Omega}
\newcommand{\Omb}{\overline{\Omega}}
\newcommand{\p}{\partial}

\newcommand{\R}{\mathbb R}

\newcommand{\ds}{\displaystyle}

\begin{document}

\title[Nonexistence for H\'enon equation]
{\bf Nonexistence of positive solutions for\\ H\'enon equation}

\author[Jorge Garc\'{\i}a-Meli\'{a}n]{Jorge Garc\'{\i}a-Meli\'{a}n}

\date{}

\address{J. Garc\'{\i}a-Meli\'{a}n \hfill\break\indent
Departamento de An\'{a}lisis Matem\'{a}tico, Universidad de La
Laguna \hfill \break \indent C/. Astrof\'{\i}sico Francisco
S\'{a}nchez s/n, 38200 -- La Laguna, SPAIN\hfill\break\indent
{\rm and} \hfill\break
\indent Instituto Universitario de Estudios Avanzados (IUdEA) en F\'{\i}sica
At\'omica,\hfill\break\indent Molecular y Fot\'onica,
Universidad de La Laguna\hfill\break\indent C/. Astrof\'{\i}sico Francisco
S\'{a}nchez s/n, 38200 -- La Laguna, SPAIN} \email{{\tt
jjgarmel@ull.es}}


\begin{abstract}
We consider the semilinear elliptic equation
$$
-\De u = |x|^\al u^p \quad \hbox{in } \R^N,
$$
where $N\ge 3$, $\al>-2$ and $p>1$. We show that there are no 
positive solutions provided that the exponent $p$ additionally verifies
$$
1<p<\frac{N+2\al+2}{N-2}.
$$
This solves an open problem posed in previous literature, where only the radially 
symmetric case was fully understood. We also characterize all positive solutions 
when $p=\frac{N+2\al+2}{N-2}$ and $-2<\al<0$.
\end{abstract}

\maketitle


\section{Introduction}
\setcounter{section}{1}
\setcounter{equation}{0}

Probably the most well-known nonlinear Liouville theorem in the literature is the one obtained in the celebrated 
paper \cite{GS}. There, nonexistence of positive solutions of the elliptic equation
\begin{equation}\label{eq-gs}
-\De u = u^p \quad \hbox{in } \R^N
\end{equation}
is established, provided that $N\ge 3$ and the exponent $p$ verifies the `subcriticality' condition
\begin{equation}\label{subcritico-gs}
1<p<\frac{N+2}{N-2}.
\end{equation}
This nonexistence theorem can be complemented with a classification result when $p=\frac{N+2}{N-2}$. 
It was proved in \cite{CGS} that every positive solution of \eqref{eq-gs} for this value of $p$ is of the form 
\begin{equation}\label{bubble-gs}
u(x)= (N(N-2)) ^\frac{N-2}{4} \hspace{-1mm} \left(\frac{\mu}{\mu^2+ |x-x_0|^{2}}\right)^\frac{N-2}{2},
\end{equation}
for some $x_0\in \R^N$ and $\mu>0$. See also \cite{LZ} for a simpler proof.

The existence of multiple applications of these results (starting with  \cite{GS2} in the context of a 
priori bounds) and its intrinsic interest have led the community to the search on one hand for simpler proofs 
(cf. \cite{BVV}, \cite{ChL}, \cite{QS}) and on the other for generalizations (see \cite{B} and \cite{LZ}).

One of the generalizations corresponds to the sometimes called H\'enon equation, namely:
\begin{equation}\label{problema}
-\De u =|x|^\al u^p  \quad \hbox{in } \R^N,
\end{equation}
where $N\ge 3$,  $p>1$ and the parameter $\al$ is an arbitrary real number. To begin with, it can always be 
assumed that $\al>-2$, since when $\al\le -2$, there are no solutions of \eqref{problema} in any 
punctured neighborhood of $x=0$ (see for instance Theorem 2.3 in \cite{DDG}).

A complete analysis of problem \eqref{problema} does not seem to have been performed, at the best of our 
knowledge. However, restricting the attention to radially symmetric solutions, it has been shown in 
\cite{BVG} that the nonexistence range of positive solutions is exactly
\begin{equation}\label{subcritico}
1<p<\frac{N+2\al+2}{N-2}.
\end{equation}
This has led (see for instance \cite{PS}) to the statement of the following:

\medskip
\begin{quotation}
{\bf Conjecture A}. Assume $\al>-2$ and $p$ verifies \eqref{subcritico}. Then problem \eqref{problema} 
does not admit any positive solution.
\end{quotation}

\medskip

\noindent Conjecture A has been proved in the case $\al<0$ in \cite{B} and for bounded solutions in dimension 
$N=3$ in \cite{PS}.  There are also some partial results, aside the just cited works. First of all, nonexistence of 
solutions with a further restriction on $p$ is a consequence of the general nonexistence result for 
supersolutions obtained in \cite{MP} (see 
also Corollary 4.2 in \cite{AS}). It is shown there that when $p$ verifies
\begin{equation}
1<p\le \frac{N+\al}{N-2}
\end{equation}
then no positive supersolutions of \eqref{problema} exist. On the other hand, the case $\al\ge 2$ is covered 
in \cite{GS}, with the further restriction \eqref{subcritico-gs}. The same restriction is found in \cite{BVV}, where 
a general $\al>-2$ is allowed. The most general result for $\al>0$ known to us for the moment is the one 
obtained in \cite{B}, where nonexistence of positive solutions was shown to hold provided that 
$$
1<p\le \frac{N+\al+2}{N-2}.
$$
But, as far as we know, when $\al>0$ and solutions are not necessarily bounded, the 
full range \eqref{subcritico} is still not covered, hence Conjecture A remains unsolved for the moment.

\medskip

Another interesting question concerning \eqref{problema} arises when the `critical' case $p=\frac{N+2\al+2}{N-2}$ 
is considered. It can be checked that, for every $\mu>0$, the functions
\begin{equation}\label{bubble}
u(x)=\left( \frac{4}{(\al+2)^2} \right) ^{-\frac{N-2}{4}} \hspace{-3mm} (N(N-2)) ^\frac{N-2}{4} \hspace{-1mm}
\left(\frac{\mu}{\mu^2+ |x|^{2+\al}}\right)^\frac{N-2}{2+\al}
\end{equation}
are solutions of \eqref{problema}. Moreover, when $\al=0$, these reduce to the corresponding ones 
for \eqref{eq-gs} taking $x_0=0$ in \eqref{bubble-gs}. Of course, when $\al\ne 0$, solutions which are 
radially symmetric with respect to a point $x_0\ne 0$ are not possible, and one may wonder whether there is a 
similar classification as that in \cite{CGS} for problem \eqref{problema}. 

The answer to this question is negative. It was shown in Theorem 1.6 of \cite{GGN} that 
for every positive even integer $\al_0$, there exists a continuum of solutions $\{(\al,u_\al) \}$ of \eqref{problema} 
with $p=\frac{N+2\al+2}{N-2}$ which are not radially symmetric, and bifurcate from $\{(\al_0,U_{\al_0})\}$, 
where $U_{\al_0}$ is a radially symmetric solution of \eqref{problema} with $\al=\al_0$ and 
$p=\frac{N+2\al_0+2}{N-2}$. Thus it is likely that non 
radially symmetric solutions exist at least for all large positive values of $\al$.

When $-2<\al<0$, however, the situation is expected to be different: it was shown in \cite{DEL} that positive 
solutions $u$ of \eqref{problema} in the critical case are given by \eqref{bubble} if $|x|^\al u^{p+1}\in L^1(\R^N)$. 
Thus it makes sense to pose the following

\medskip
\begin{quotation}
{\bf Question B}. Assume $-2<\al<0$ and $p=\frac{N+2\al+2}{N-2}$. Are all positive solutions of \eqref{problema} 
of the form \eqref{bubble}?
\end{quotation}

\medskip

We come now to the statement of our results. By a solution of \eqref{problema}, we mean a function 
$u\in H^1_{\rm loc}(\R^N)\cap L^\infty_{\rm loc}(\R^N)$, verifying the equation in the weak sense. 
However, it is worthy of mention that with the aid of the results in \cite{PS} and \cite{BVV}, it suffices 
to assume that $u\in H^1_{\rm loc}(\R^N \setminus \{0\})\cap L^\infty_{\rm loc}(\R^N\setminus \{0\})$ verifies the 
equation in the weak sense in $\R^N\setminus \{0\}$, together with the condition
$$
\lim_{x\to 0} |x|^{\frac{2+\al}{p-1}} u(x)=0.
$$
We will show in our first result that Conjecture A holds in the full regime \eqref{subcritico} and for all 
values $\al>-2$, therefore providing a proof which unifies both cases $\al<0$ and $\al>0$. 

\medskip

\begin{teorema}\label{th-liouville}
Assume $N\ge 3$, $\al>-2$ and $p$ verifies \eqref{subcritico}. 
Then problem \eqref{problema} does not admit any positive solution.
\end{teorema}

\bigskip

Moreover, in our second result we also answer affirmatively Question B in the full 
regime $-2<\al<0$.

\medskip

\begin{teorema}\label{th-clasif}
Assume $N\ge 3$, $-2<\al <0$ and 
\begin{equation}\label{critico}
p=\frac{N+2\al+2}{N-2}.
\end{equation}
Let $u \in H^1_{\rm loc}(\R^N)\cap L^\infty_{\rm loc}(\R^N)$ be a positive solution of 
\eqref{problema}. Then $u$ is of the form \eqref{bubble} for some $\mu>0$.
\end{teorema}

\bigskip

The proofs of Theorems \ref{th-liouville} and \ref{th-clasif} rely on the 
well-known trick of writing the equation in polar coordinates, and then introducing the Emden-Fowler 
transformation. An application of the moving planes method as in \cite{BM} yields the 
monotonicity of the function 
$$
|x|^{\frac{2+\al}{p-1}} u(x)
$$
in $\R^N$, provided that \eqref{subcritico} holds. In the case $\al=0$, the same argument 
applied with respect to an arbitrary origin then shows that $u$ has to be constant, which is not 
possible, thus we obtain a simplified proof of the Liouville theorem in \cite{GS}. 
But this argument does not carry over to deal with $\al\ne 0$. However, the essential point in 
our proof is to realize that the monotonicity alluded to above shows that
$$
u \hbox{ is a stable solution of }\eqref{problema},
$$
in the usual sense that for every $\phi \in C_0^\infty(\R^N)$ there holds
\begin{equation}\label{cond-estab}
\int_{\R^N} |\nabla \phi|^2 - p |x|^\al u^{p-1} \phi^2 \ge 0.
\end{equation}
A Liouville theorem for stable solutions of \eqref{problema} is already available 
(cf. \cite{DDG} and the previous work \cite{F} for problem \eqref{eq-gs}) and it implies $u\equiv 0$ in $\R^N$, 
a contradiction. For the reader's convenience, we include an independent, simplified proof of 
this theorem. 

As for Theorem \ref{th-clasif}, we obtain as a byproduct of the same moving plane argument that 
$u$ behaves at infinity like the fundamental solution of the Laplacian, hence the results in \cite{DEL} 
can be used to obtain that $u$ is of the form \eqref{bubble}.

\medskip

Finally, it is interesting to mention that the approach followed to prove Theorem \ref{th-liouville} 
can be used to deal with other related problems. For instance, when the weight $|x|^\al$ is replaced 
by $|x_1|^\al$ or even more general functions. In this regard, the problem 
\begin{equation}\label{problema-2}
-\De u = x_1^m u^p \quad \hbox{in } \R^N,
\end{equation}
where $m$ is a positive integer, has been already considered in previous literature. We refer to 
\cite{BCN}, \cite{L}, \cite{DL}. However, in all these works only \emph{odd} integers are allowed. 
Our methods enable us to obtain a Liouville theorem in the complementary case where $m$ is an 
even integer. The proof of the following result is a slight variant of that of Theorem \ref{th-liouville} 
and will not be given.

\begin{teorema}
Assume $N\ge 3$, $m$ is an even integer and 
$$
1<p<\frac{N+2m+2}{N-2}.
$$
Then problem \eqref{problema-2} does not admit any positive solution.
\end{teorema}

\medskip

The rest of the paper is organized as follows: in Section 2 we include some preliminaries related
with regularity of solutions, principal eigenvalues in smooth bounded domains with unbounded 
coefficients and the Liouville theorem for stable solutions of \eqref{problema}. 
Section 3 is dedicated to the proof of Theorems \ref{th-liouville} and \ref{th-clasif}.

\bigskip

\section{Preliminaries}
\setcounter{section}{2}
\setcounter{equation}{0}

In this section we will consider some preliminaries on positive solutions of \eqref{problema}. 
Most of them deal with regularity, especially in the case $\al<0$. We will also briefly deal with an 
eigenvalue problem with coefficients which are not bounded and we will include a proof of the nonexistence 
of stable positive solutions of \eqref{problema} when $1<p\le \frac{N+2\al+2}{N-2}$.

The first result is a consequence of standard regularity theory.

\begin{lema}\label{regularidad}
Let $u\in H^1_{\rm loc}(\R^N)\cap L^\infty_{\rm loc}(\R^N)$ be a positive weak solution of \eqref{problema}. 
Then:
\begin{itemize}
\item[(a)] If $-2<\al <0$, we have $u\in C^\infty(\R^N\setminus \{0\}) \cap W^{2,q}_{\rm loc}(\R^N)$ for 
some $q>\frac{N}{2}$. Moreover, there exists $\eta\in (0,1)$ such that $u\in C^\eta(\R^N)$.
In addition, there exists $C>0$ such that 
\begin{equation}\label{des-gradiente}
|\nabla u(x) |\le \frac{C}{|x|} \quad \hbox{if } 0<|x| <1.
\end{equation}

\item[(b)] When $\al\ge 0$, $u\in C^\infty(\R^N\setminus \{0\}) \cap C^{2,\eta}(\R^N)$ for some 
$\eta\in (0,1)$. 
\end{itemize}
\end{lema}

\begin{proof}
The assertion about $C^\infty$ regularity in $\R^N \setminus \{0\}$ is immediate by bootstrapping and the fact that 
$|x|^\al$ is $C^\infty$ there, while $u$ is positive (cf. \cite{GT}). 

The regularity in the case $\al\ge 0$ is a consequence of standard theory: we obtain that $\De u$ 
is locally bounded in $\R^N$, so that $u\in C^1(\R^N)$. This 
implies in turn that $\De u \in C^{\eta}(\R^N)$ for some $\eta\in (0,1)$, so that $u \in C^{2,\eta}(\R^N)$. 
Observe that solutions become more regular the larger $\al$ is.

When $-2<\al <0$, we see that $|x|^\al \in L^q_{\rm loc}(\R^N)$ for every $q<\frac{N}{|\al|}$. Thus we 
may choose and fix such a $q$ additionally verifying $q>\frac{N}{2}$. This implies that 
$h=|x|^\al u^{p} \in L^q_{\rm loc}(\R^N)$. Denoting the unit ball of 
$\R^N$ by $B$, we can use Theorem 9.15 in \cite{GT} to guarantee that the problem
$$
\left\{
\begin{array}{ll}
-\De w = h & \hbox{in }B\\
\ \ \ w=u & \hbox{on } \p B
\end{array}
\right.
$$
admits a unique strong solution $w\in W^{2,q}(B)$. Since $q>\frac{N}{2}$ it also follows by Sobolev 
embeddings that $w\in C^\eta(\R^N)$ for some $\eta\in (0,1)$. We deduce that $\De (u-w)=0$ in 
$B\setminus \{0\}$, while $u-w$ is 
bounded in $B$. It is well-known that this implies $u \equiv w$ in $B$. 
We conclude that $u\in W^{2,q}_{\rm loc}(\R^N)\cap C^\eta(\R^N)$.

To show \eqref{des-gradiente}, we make use once more of standard regularity. Fix $x\in B\setminus 
\{0\}$ and consider the ball $B_x$ with center $x$ and radius $\frac{|x|}{2}$. There exists a positive constant 
$C$ which does not depend on $x$ nor on $u$ such that
\begin{equation}\label{est-grad-2}
|x|  |\nabla u(y)| \le C (|x|^2 \sup_{z \in B_x}  |\De u(z)| + \sup_{z \in B_x} |u(z)|)
\end{equation}
for every $y\in B_x$ (cf. for instance (4.45) in \cite{GT}). Observe that $|z|\ge \frac{|x|}{2}$ for 
every $z\in B_x$, so that
$$
|x|^2 |\De u(z)| \le C |x| ^2 |z|^\al \le C |x|^{2+\al}\le C.
$$
Thus taking $y=x$ in \eqref{est-grad-2} we obtain \eqref{des-gradiente}.
\end{proof}

\bigskip

Next we consider a special solution of the linearized equation, which is one of the keys to 
our proofs in Section 3. Also, we need to `fine tune' the regularity of the gradient of the 
solutions near $x=0$. Throughout the rest of the paper, we will denote
\begin{equation}\label{def-beta}
\beta=\frac{2+\al}{p-1}.
\end{equation}

\begin{lema}\label{lema-solucion}
Assume $u\in H^1_{\rm loc}(\R^N)\cap L^\infty_{\rm loc}(\R^N)$ is a positive weak solution of 
\eqref{problema}. Then the function 
$v(x)= \nabla u(x) \cdot x + \be u(x)$ belongs to $C^\infty (\R^N\setminus \{0\}) \cap 
W^{2,q}_{\rm loc}(\R^N)$ for some $q>\frac{N}{2}$ and verifies
$$
-\De v = p |x|^{\al} u^{p-1} v \quad \hbox{in } \R^N\setminus \{0\}.
$$
Moreover $v\in C^\eta(\R^N)$ for some $\eta \in (0,1)$ and in particular
\begin{equation}\label{lim-grad-cero}
\lim_{x\to 0} \nabla u(x) \cdot x=0.
\end{equation}
\end{lema}

\begin{proof}
Let $v(x)= \nabla u(x) \cdot x + \be u(x)$. Then it is not hard to see that 
\begin{align*}
-\De v & = -\nabla (\De u) \cdot x - (\be+2)\De u\\
& =(\al+\be+2) |x|^\al u^p + p |x|^\al u^{p-1} \nabla u\cdot x\\
& =(\al+\be+2-\be p) |x|^\al u^p + p |x|^\al u^{p-1} v\\
& = p |x|^\al u^{p-1} v
\end{align*}
in $\R^N\setminus \{0\}$. 

On the other hand, it is clear from Lemma \ref{regularidad} that $v\in C^{1,\eta}(\R^N)$ 
when $\al	\ge 0$, while $v\in L^\infty_{\rm loc}(\R^N)$ if $\al<0$ (cf. in particular equation \eqref{des-gradiente}). 
Then $h=p|x|^{\al} u^{p-1}v \in L^q_{\rm loc}(\R^N)$ for some $q> \frac{N}{2}$, and reasoning 
as in the proof of Lemma \ref{regularidad} we deduce $v\in W^{2,q}_{\rm loc}(\R^N)$. The 
$C^\infty$ regularity in $\R^N\setminus \{0\}$ is immediate from Lemma \ref{regularidad}.

From Sobolev embeddings,  we also have $v\in C^\eta(\R^N)$ for some $\eta \in (0,1)$. 
We deduce that $\nabla u(x)\cdot x$ is continuous at zero, so that the limit
$$
\ell =\lim_{x\to 0} \nabla u(x) \cdot x
$$
exists. Since $u$ is bounded at zero, this limit has to be zero. This concludes the proof.
\end{proof}

\bigskip

It is the turn now to consider an auxiliary eigenvalue problem. In what follows, we deal with 
a smooth bounded domain $\Om$ of $\R^N$, and a coefficient $a\in L^q(\Om)$, where 
$q>\frac{N}{2}$. We are interested in the principal eigenvalue of
\begin{equation}\label{prob-autov}
\left\{
\begin{array}{ll}
-\De u +a(x) u = \la u & \hbox{in }\Om\\
\ \ \ u=0 & \hbox{on } \p\Om,
\end{array}
\right.
\end{equation}
that is, the first of the eigenvalues, which is associated to a positive eigenfunction. Although 
we expect the next result to be well-known, we have not been able to find a pertinent reference. 
We refer the reader to Theorem 1 in \cite{CRQ}, where the extension to the $p-$Laplacian 
setting is analyzed.

\begin{lema}\label{autovalores}
Assume $\Om\subset \R^N$ is a smooth bounded domain, and let $a \in L^q(\Om)$ for 
some $q> \frac{N}{2}$. Then problem \eqref{prob-autov} admits a principal eigenvalue 
$\la_1^\Om(a)$, which can be variationally characterized as
\begin{equation}\label{variacional}
\la_1^\Om(a) = \inf_{w\in H_0^1(\Om)} \frac{\ds \int_\Om |\nabla w|^2 + a(x) w^2}{\ds \int_\Om w^2}.
\end{equation}
Moreover, there exists an associated positive eigenfunction $\phi \in H_0^1(\Om)\cap W^{2,q}(\Om)
\cap C^\eta(\Omb)$, for some $\eta \in (0,1)$.
\end{lema}

\bigskip

We deal next with a very well-known property of the principal eigenvalue. 
When the coefficient $a$ is bounded, $\la_1^\Om(a)$ can be characterized as:
$$
\la_1^\Om(a)= \sup\left\{\la>0:\ \begin{array}{c} \hbox{there exists }v>0 \hbox{ in }\Omb \hbox{ such that }\\ 
-\De v +a(x) v \ge \la v \ \hbox{ a. e. in } \Om\end{array}\right\},
$$
while the functions $v$ are taken in $W^{2,N}(\Om)$ (see for instance \cite{BNV} or a more recent 
account in unbounded domains in \cite{BR}). In particular, the existence of a positive function $v\in W^{2,N}(\Om)$ 
verifying $-\De v +a(x) v\ge 0$ in $\Om$ implies $\la_1^\Om (a)>0$.

We are not aware of any similar property when the coefficient $a$ is not bounded, or when the 
function $v$ is not in $W^{2,N}(\Om)$. Thus we obtain one which is sufficient for our purposes in 
Section 3.

\begin{lema}\label{autov-positivo}
Assume $\Om\subset \R^N$ is a smooth bounded domain and let $a \in L^q(\Om)$ for 
some $q> \frac{N}{2}$. If there exists $v\in W^{2,q}(\Om)\cap C(\Omb)$ such that $v>0$ in $\Omb$ 
and $-\De v +a(x) v \ge 0$ a. e. in $\Om$, then 
$$
\la_1^\Om(a)>0.
$$
\end{lema}

\begin{proof}
The proof is based on standard arguments, the only important point being the  use of a strong 
maximum principle for $W^{2,q}$ functions with $q> \frac{N}{2}$, available thanks to  
Theorem 1 in \cite{OP} or Corollary 5.1 in \cite{T}. 

Assume for a contradiction that $\la_1^\Om (a)\le 0$, and let $\phi \in H_0^1 (\Om)\cap 
W^{2,q}(\Om)\cap C(\Omb)$ be a positive eigenfunction given by Lemma \ref{autovalores}. Since 
$v\in C(\Omb)$ is positive, we have
$$
\gamma= \inf_{x\in \Om} \frac{v(x)}{\phi(x)}>0.
$$
Consider the function $z= v-\gamma \phi$. It is clear that $z\in W^{2,q}(\Om)\cap C(\Omb)$, while 
$z\ge 0$ in $\Om$. By continuity, and since $v>0$ on $\p \Om$ while $\phi=0$ there, there exists 
$x_0\in \Om$ such that $z(x_0)=0$. Moreover, 
$$
-\De z + a(x) z \ge -\la_1^\Om (a) \gamma \phi \ge 0 \quad \hbox{in } \Om.
$$
We may use the strong maximum principle to conclude 
that $z\equiv 0$ in $\Om$, which is not possible because $v>0$ on $\p\Om$ and $\phi$ vanishes on 
$\p\Om$.  Therefore $\la_1^\Om (a)>0$, as we wanted to show.
\end{proof}

\bigskip

To conclude the section, we consider the Liouville theorem for stable solutions of \eqref{problema}. 
It is worthy of mention that the nonexistence result in \cite{DDG} (Theorem 1.2 there) is more general, but we 
restrict ourselves to the subcritical range of the parameter, which allows us to give a simpler proof.

\begin{teorema}\label{th-liouville-est}
Assume $N\ge 3$, $\al>-2$ and 
\begin{equation}\label{cond-subcr}
1<p\le \frac{N+2\al+2}{N-2}.
\end{equation}
Then the unique stable solution of \eqref{problema} is $u\equiv 0$.
\end{teorema}

\begin{proof}
Let $\varphi\in C^\infty_0(\R^N)$ be arbitrary. Taking $\varphi^2 u$ as a test function in 
\eqref{problema} and setting $\phi=\varphi u$ in \eqref{cond-estab}, we obtain
\begin{equation}\label{eq-estab-1}
(p-1) \int |x|^\al \varphi^2 u^{p+1} \le \int u^2 |\nabla \varphi|^2.
\end{equation}
Now choose $\xi\in C_0^\infty(B_2)$ such that $0\le \xi\le 1$ and $\xi\equiv 1$ in $B_1$. 
Take $R>0$ and choose
$$
\phi(x)= \xi\left(\frac{x}{R}\right)^\frac{p+1}{p-1}
$$
in \eqref{eq-estab-1}. It is easily seen that this implies
$$
\int_{B_{2R}} |x|^\al \varphi^2 u^{p+1} \le \frac{C}{R^2} \int_{B_{2R}\setminus B_R} \varphi^\frac{4}{p+1}u^2,
$$
for some $C>0$ (from now on, we are using the letter $C$ to denote different constants not depending on $R$). 
Using H\"older's inequality in the last integral with conjugate exponents 
$\frac{p+1}{2}$ and $\frac{p+1}{p-1}$ yields
\begin{equation}\label{eq-estab-3}
\begin{array}{rl}
\ds \int_{B_{2R}} |x|^\al \varphi^2 u^{p+1} \hspace{-2mm} & \ds \le C R^{-2+N\frac{p-1}{p+1}} 
\left(\int_{B_{2R}\setminus B_R} \varphi^2 u^{p+1} \right)^\frac{2}{p+1}\\[.5pc]
& \ds \le  C R^{-2-\frac{2\al}{p+1}+N\frac{p-1}{p+1}} 
\left(\int_{B_{2R}\setminus B_R} |x|^\al \varphi^2 u^{p+1} \right)^\frac{2}{p+1}.
\end{array}
\end{equation}
As a consequence we arrive at 
\begin{equation}\label{eq-estab-2}
\int_{B_{R}} |x|^\al u^{p+1} \le \int_{B_{2R}} |x|^\al \varphi^2 u^{p+1} \le 
C R^\frac{N(p-1)-2(p+1)-2\al}{p-1}.
\end{equation}
It is not hard to check that, when the second inequality in \eqref{cond-subcr} is strict, the exponent of $R$ 
in \eqref{eq-estab-2} is negative. Hence letting $R\to +\infty$ we see that $u\equiv 0$. 

In the case where $p=\frac{N+2\al+2}{N-2}$, inequality \eqref{eq-estab-2} gives $|x|^\al u^{p+1}\in L^1(\R^N)$. 
Thus letting $R\to +\infty$ in \eqref{eq-estab-3} we also obtain $u\equiv 0$ in $\R^N$. The proof 
is concluded. 
\end{proof}

\bigskip

\section{Proof of the main results}
\setcounter{section}{3}
\setcounter{equation}{0}

This section is devoted to prove Theorems \ref{th-liouville} and \ref{th-clasif}. As we have 
mentioned in the Introduction, the fundamental step in both theorems is to obtain a monotonicity 
property of the solutions. It is worthy of mention that in this case, the restriction 
\begin{equation}\label{rest-exp}
p>\frac{N+\al}{N-2} 
\end{equation}
is important, and does not imply any loss in generality. 
The idea for the proof of the next result comes from \cite{BM}. Recall our definition \eqref{def-beta}.

\begin{teorema}\label{monotonia}
Assume $\al>-2$ and $p$ verifies \eqref{rest-exp}. Let $u\in H^1_{\rm loc}(\R^N)\cap L^\infty_{\rm loc}(\R^N)$
be a positive weak solution of \eqref{problema}. Then:

\begin{itemize}

\item[(a)] When \eqref{subcritico}  holds, the function $|x|^\beta u(x)$ is nondecreasing in $|x|$.

\smallskip
\item[(b)] If \eqref{critico} holds, either $|x|^\be u(x)$ is nondecreasing in $|x|$ or 
there exists $\mu>0$ such that 
\begin{equation}\label{kelvin}
u(x)= \mu ^{N-2}|x|^{2-N} u\left( \mu^2 \frac{ x}{|x|^2}\right), \quad x\in \R^N \setminus \{0\},
\end{equation}
that is, $u$ coincides with its Kelvin transform with respect to some ball centered at the origin.
\end{itemize}
\end{teorema}

\begin{proof}
To start with, we consider polar coordinates, and write $u(x)=v(r,\theta)$, where $r=|x|$ and 
$\theta =\frac{x}{|x|}\in S^{N-1}$. It is not hard to see that $v$ verifies
$$
v_{rr} + \frac{N-1}{r} v_r + \frac{1}{r^2} \De_\theta v + r^\al v^p =0 \quad  \hbox{in }(0,+\infty)\times S^{N-1},
$$
in the classical sense, 
where $\De_\theta$ stands for the Laplace-Beltrami operator on $S^{N-1}$. We now introduce the
Emden-Fowler transformation,
$$
w(t,\theta)= r^\beta v(r,\theta) \quad \hbox{where } t= \log r.
$$
After a straightforward calculation we see that $w$ verifies
$$
w_{tt} +aw_t + \De_\theta w- b w + w^p=0 \quad  \hbox{in } \R \times S^{N-1},
$$
where
\begin{equation}\label{def-coef}
\begin{array}{l}
a= 2\be- (N-2) \ge 0\\[0.25pc]
b= \be (N-2-\be)>0.
\end{array}
\end{equation}
Our proof reduces to show that $w$ is nondecreasing in the variable $t$ in $\R$.
For this sake we employ the moving planes method (cf. \cite{GNN}, \cite{BN}). We follow the standard notation:
$$
\begin{array}{l}
\Sigma_\la=(-\infty,\la)\times S^{N-1}\\[0.25pc]
T_\la =\{\la\} \times S^{N-1}\\[0.25pc]
t^\la=2\la-t \quad \hbox{if } t<\la,
\end{array}
$$
and let 
$$
z^\la (t,\theta) = w(t^\la,\theta)- w(t,\theta), \quad (t,\theta) \in \Sigma_\la.
$$
By the mean value theorem, there exists $\xi_\la=\xi_\la(t,\theta)$ such that 
$w(t^\la,\theta)^p- w(t,\theta)^p= p \xi_\la^{p-1} z^\la$. 
Then the function $z^\la$ verifies the equation
\begin{equation}\label{eq-simetria}
z^\la_{tt} - a z^\la_t +\De_\theta z^\la - b z^\la + p\xi_\la^{p-1}z^\la = -2a w_t \quad \hbox{in }\Sigma_\la.
\end{equation}
We next observe that the boundedness of $u$ at $x=0$ implies 
\begin{equation}\label{lim-cero}
\lim_{t\to -\infty} \inf_{\theta\in S^{N-1}} w(t,\theta)=0 \qquad \hbox{ and } \qquad
\lim_{t\to -\infty} \inf_{\theta\in S^{N-1}} z^\la(t,\theta) \ge 0.
\end{equation}
Moreover,
$$
w_t= r^\be (\be u + r u_r),
$$
so that from \eqref{lim-grad-cero} in Lemma \ref{lema-solucion} we see that there exists $\bar t\ll -1$ 
such that $w_t>0$ in $(-\infty,\bar t)\times S^{N-1}$. Diminishing $\bar t$ if necessary, we can also 
achieve by \eqref{lim-cero} that
$$
-b+ p w^{p-1}\le -\frac{b}{2}<0 \quad \hbox{for } t\le \bar t.
$$

\medskip
\noindent \underline{Claim}: $z^\la \ge 0$ in $\Sigma_\la$ when $\la\le \bar t$. 

\smallskip
Indeed, assume on the contrary that 
$$
\inf_{\Sigma_\la} z^\la <0
$$
for some $\la \le \bar t$. 
By \eqref{lim-cero} and since $z^\la=0$ on $T_\la$, we deduce the existence of a point 
$(t_0,\theta_0)\in \Sigma_\la$ such that the infimum of $z^\la$ is achieved. Thus
$z^\la_t(t_0,\theta_0)=0$, $z^\la_{tt}(t_0,\theta_0)\ge 0$, $\Delta_\theta z^\la (t_0,\theta_0)\ge 0$ and 
$$
-b+p\xi_\la(t_0,\theta_0)^{p-1} \le -b+p w(t_0,\theta_0)^{p-1} \le - \frac{b}{2}
$$ 
by our choice of $\bar t$. Then, using \eqref{eq-simetria} and recalling \eqref{def-coef}:
$$
(-b+   p\xi_\la(t_0,\theta_0)^{p-1} ) z^\la (t_0,\theta_0) \le -2a w_t(t_0,\theta_0)\le 0,
$$
a contradiction. This shows the claim. 

\medskip

Next define
$$
\la_0= \sup \{\la \in \R:\ z^\mu\ge 0  \hbox{ in } \Sigma_\mu,  \hbox{ for every }\mu \in (-\infty,\la)\}.
$$
Two situations are possible:

\smallskip
\begin{itemize}
\item[(i)] $\la_0=+\infty$;

\medskip
\item[(ii)] $\la_0<+\infty$.
\end{itemize}

\smallskip
\noindent In case (i), we would obtain that $z^\la\ge 0$ in $\Sigma_\la$ for every $\la\in \R$, which implies 
that $w$ is nondecreasing in the $t$ variable.

\smallskip

Therefore we only have to deal with case (ii). By the strong maximum principle we deduce that 
either  $z^{\la_0}\equiv 0$ in $\Sigma_{\la_0}$ or $z^{\la_0}>0$ in $\Sigma_{\la_0}$ with $z_t<0$ on 
$\{ \la_0 \} \times S^{N-1}$. 

The second situation can be easily discarded. By the definition of $\la_0$, there exist sequences 
$\la_n\to \la_0+$, $(t_n,\theta_n)\in \Sigma_{\la_n}$ such that
$$
z^{\la_n} (t_n,\theta_n)<0.
$$
We claim that we can always assume $t_n\ge \bar t$. Otherwise, we would have $z^{\la_n}\ge 0$ 
in $\Sigma_{\la_n}\setminus \Sigma_{\bar t}$. Reasoning as in the claim above this would yield $z^{\la_n}\ge 0$ 
in $\Sigma_{\la_n}$, a contradiction. 

Therefore $t_n\in [\bar t,\la_n]$. Passing to subsequences we may assume $t_n\to t^* \in [\bar t, \la_0]$, 
$\theta_n\to \theta^* \in S^{N-1}$. Then $z^{\la_0}(t^*,\theta^*)=0$. Let us see that this is impossible. 

If $t^*<\la_0$ we have an immediate contradiction with $z^{\la_0}>0$ in $\Sigma_{\la_0}$. 
When $t^*=\la_0$, we can select points $s_n\in (t_n,\la_n)$ such that $z_t(s_n,\theta_n)\ge 0$. 
Passing to the limit this would yield $z_t(\la_0, \theta^*)\ge 0$, which is also a contradiction.

\medskip

To summarize, we have shown that when $\la_0<+\infty$ we always have $z^{\la_0}\equiv 0$ in $\Sigma_{\la_0}$, 
that is, $w$ is symmetric with respect to $\la_0$. 
When \eqref{subcritico} holds we have that the coefficient $a$ in \eqref{def-coef} is strictly positive. The symmetry 
of $w$ would imply from \eqref{eq-simetria} that $w_t=0$ in $\R\times S^{N-1}$, which is equivalent to 
$w=w_0(\theta)$ for some positive function $w_0$. Then $u(x) = w_0(\frac{x}{|x|}) |x|^{-\beta}$, which contradicts the boundedness of $u$ near $x=0$. This shows that, with the subcriticality assumption \eqref{subcritico} 
we always have $\la_0=+\infty$, and (a) is proved.

When $p$ verifies \eqref{critico}, both $\la_0=+\infty$ and $\la_0<+\infty$ are possible. In the second 
case, the function $w$ is 
symmetric with respect to $\la_0$. Setting $\mu^2 = e^{2\la_0}$ and rewriting the symmetry property of $w$ in 
terms of the original function $u$ we obtain \eqref{kelvin}. This concludes the proof of (b).
\end{proof}
 
\bigskip

After all these preliminaries, we are in a position to prove our two main results, Theorems \ref{th-liouville} 
and \ref{th-clasif}.

\bigskip

\begin{proof}[Proof of Theorem \ref{th-liouville}]
Because of the already mentioned results in \cite{AS} and \cite{MP}, 
we may assume that $p$ verifies \eqref{rest-exp}. 
We claim that $u$ is stable in $\R^N$. To see this, we first apply Theorem \ref{monotonia} to 
obtain that $|x|^\be u(x)$ is nondecreasing as a function of $|x|$. Since 
$u$ is smooth in $\R^N\setminus \{0\}$ by Lemma \ref{regularidad}, this implies that 
$$
v=\nabla u(x) \cdot x + \be u(x) \ge 0 \quad \hbox{in } \R^N\setminus \{0\}.
$$
Moreover, from Lemma \ref{lema-solucion} we see that $v\in W^{2,q}_{\rm loc}(\R^N)$ is a 
solution of the linearized equation
$$
-\De v = p |x|^{\al} u^{p-1} v \quad \hbox{in } \R^N\setminus \{0\}.
$$
The strong maximum principle gives either $v>0$ in $\R^N$ or $v\equiv 0$ in $\R^N$. 
However, this last option may not occur, since it would imply that $u$ is homogeneous of 
degree $\beta$, contradicting its boundedness at zero.

Therefore $v>0$ in $\R^N$. By Lemma \ref{autov-positivo}, this implies $\la_1 ^\Om(-p |x|^\al u^{p-1})>0$, 
for every smooth bounded domain $\Om\subset \R^N$. 
Using the variational characterization \eqref{variacional}, we arrive at 
$$
\int_\Om |\nabla \phi|^2 - p |x|^\al u^{p-1} \phi^2 \ge \la_1^\Om (-p |x|^\al u^{p-1})
\int_\Om \phi^2 \ge 0,
$$
for every $\phi \in C^\infty_0(\Om)$. Since $\Om$ is arbitrary, the stability of $u$ 
is established. To conclude the proof, we use Theorem \ref{th-liouville-est}, which implies 
that $u\equiv 0$, a contradiction.
\end{proof}

\bigskip

\begin{proof}[Proof of Theorem \ref{th-clasif}]
By Theorem \ref{monotonia} and the proof of Theorem \ref{th-liouville}, we see that 
\eqref{kelvin} holds for some $\mu>0$. Observe that this implies 
\begin{equation}\label{comportamiento}
\lim_{x\to +\infty} |x|^{N-2} u(x)=\mu^{N-2} u(0).
\end{equation}
Since $-2<\al<0$, we may use the moving plane method directly on $u$ to obtain 
that $u$ is radially symmetric. Then the conclusion follows from Appendix A in \cite{GS}.

An alternative proof is to observe that \eqref{comportamiento} implies 
$$
\int_{\R^N} |x|^\al u(x)^{p+1}dx <+\infty,
$$
Theorem 1.2 in \cite{DEL} implies that $u$ is of the form \eqref{bubble} for some 
$\mu>0$.
\end{proof}

\bigskip

\noindent {\bf Acknowledgements.} Supported by Spanish Ministerio de Econom\'ia y 
Competitividad under grant MTM2014-52822-P. The author would like to thank professor 
Philippe Souplet for some useful comments on a previous version of this work.


\end{document}